\documentclass[12pt]{amsart}

\usepackage[pdfauthor   = {Mohamed\ Barakat\ and\ Markus\ Lange-Hegermann\ and\ Sebastian\ Posur},
            pdftitle    = {Elimination\ via\ saturation},
            pdfsubject  = {},
            pdfkeywords = {elimination;\ homogenization;\ saturation;\ ideal quotients;\ syzygies\; Gröbner\ bases},
            bookmarks=true,
            bookmarksopen=true,
            pagebackref=true,
            hyperindex=true,
            colorlinks=true,
            linkcolor=blue,
            citecolor=blue,
            filecolor=blue,
            urlcolor=blue,
            ]{hyperref}

\usepackage[utf8,utf8x]{inputenc}
\usepackage[english]{babel}
\usepackage[T1]{fontenc}
\usepackage{geometry}                
\usepackage{times}
\usepackage{mathrsfs}
\usepackage{latexsym}
\usepackage{amssymb}
\usepackage{amsthm}
\usepackage{mathtools} 
\usepackage{epsfig}
\usepackage{colortbl}
\usepackage[all]{xy}
\usepackage{fancyvrb}
\usepackage{graphicx}
\usepackage[dvipsnames]{xcolor}
\usepackage{accents} 
\usepackage{enumerate}

\usepackage{wrapfig}
\usepackage{tikz}
\usetikzlibrary{shapes,arrows,matrix,backgrounds,positioning,plotmarks,calc,patterns,
decorations.shapes,
decorations.fractals,
decorations.markings,
decorations.pathreplacing,
decorations.pathmorphing,
decorations.text
}
\usepackage[colorinlistoftodos,shadow]{todonotes}

\usepackage{multirow}
\usepackage{mdwlist}
\usepackage{stmaryrd}
\usepackage{mathdots} 

\usepackage[toc,page]{appendix}
\usepackage{float}

\usepackage[sort&compress,capitalise]{cleveref}

\usepackage[linesnumbered,commentsnumbered,ruled,vlined]{algorithm2e}

\SetCommentSty{mycommfont}

\newtheoremstyle{mytheoremstyle} 
    {5pt}                    
    {5pt}                    
    {\itshape}                   
    {\parindent}                           
    {\bf}                   
    {.}                          
    {.5em}                       
    {}  

\theoremstyle{mytheoremstyle}

\newtheorem{theorem}{Theorem}[section]

\newtheorem{lemm}[theorem]{Lemma}
\newtheorem{prop}[theorem]{Proposition}
\newtheorem{coro}[theorem]{Corollary}

\newtheoremstyle{mytdefintionstyle} 
    {5pt}                    
    {5pt}                    
    {\rm}                   
    {\parindent}                           
    {\bf}                   
    {.}                          
    {.5em}                       
    {}  

\theoremstyle{remark}
\newtheorem{rmrk}[theorem]{Remark}

\theoremstyle{mytdefintionstyle}

\newtheorem{exmp}[theorem]{Example}

\newtheoremstyle{exmp_contd}
    {5pt}                    
    {5pt}                    
    {\rm}                   
    {\parindent}                           
    {\bf}                   
    {.}                          
    {.5em}                       
    {\thmname{#1}\ \thmnumber{ #2}\thmnote{#3}\ (continued)}  
\theoremstyle{exmp_contd}

\newcommand\nameft\textrm

\DeclareMathOperator{\Spec}{Spec}

\newcommand{\Q}{\mathbb{Q}}

\newcommand{\Z}{\mathbb{Z}}
\newcommand\N{\mathbb{N}}
\renewcommand\phi{\varphi}

\definecolor{darkgray}{rgb}{0.3,0.3,0.3}

\pgfarrowsdeclarecombine[\pgflinewidth]
  {doublestealth}{doublestealth}{stealth'}{stealth'}{stealth'}{stealth'}

\definecolor{darkgreen}{rgb}{0.008,0.617,0.067}
\definecolor{brown}{rgb}{0.6,0.4,0.2}

\makeatletter
\@ifundefined{journal}{%
  \newif\ifjournalversion%
  
}{}
\makeatother

\author{Mohamed Barakat}
\address{Department of mathematics\\University of Siegen\\57068 Siegen, Germany}
\email{\href{mailto:Mohamed Barakat <mohamed.barakat@uni-siegen.de>}{mohamed.barakat@uni-siegen.de}}

\author{Markus Lange-Hegermann}
\address{Department of Electrical Engineering and Computer Science\\Ostwestfalen-Lippe University of Applied Sciences\\32657 Lemgo, Germany}
\email{\href{mailto:Markus Lange-Hegermann <markus.lange-hegermann@th-owl.de>}{markus.lange-hegermann@th-owl.de}}

\author{Sebastian Posur}
\address{Lehrstuhl für Algebra und Darstellungstheorie\\RWTH-Aachen University\\52062 Aachen, Germany}
\email{\href{mailto:Sebastian Posur <posur@art.rwth-aachen.de>}{posur@art.rwth-aachen.de}}

\begin{document}

\title[Elimination via saturation]{Elimination via saturation}

\begin{abstract}
This short paper presents saturation-based algorithms for homogenization and elimination.
This algorithm can compute elimination ideals by using syzygies and ideal membership test, hence it works with \emph{any} monomial order, in particular without the use of block-elimination orders.
The used saturation is a translation of the geometric fact that the projective closure of an affine scheme has no components in the hyperplane at infinity.
\end{abstract}

\maketitle


\section{Introduction}

Elimination has many applications, for example solving systems of equations, computing primary decompositions, and computing Zariski closures of images of morphisms of varieties.
Furthermore, using elimination, one can compute ideal quotients\footnote{Cf.~\cite[Exer.~15.41]{Eis} or \cite[Lemma 2.3.11]{AL}.} and saturations of ideals, either by iterating ideal quotients or using ``Rabinowitch's trick''\footnote{Cf.~\cite[Exer.~2.3.14.c]{AL}.}.
This leads us to our first question: \\
\textbf{Q1.} \emph{Is the converse true, i.e., can one use saturation to eliminate?}

Eisenbud writes in \cite[Chapter~15]{Eis}: ``The kinds of problems that can be attacked with Gröbner bases can be very roughly divided into two groups: constructive module theory and elimination theory.''
This motivates our second question: \\
\textbf{Q2.}
\emph{Is constructive elimination theory a special case of constructive module theory?}

In elimination theory one computes Gröbner bases w.r.t.\  a block-elimination order.
For the module theoretic applications any global order can be used, e.g, the degree-reverse-lexicographic order.
So the second question can be reformulated as follows: \\
\textbf{Q2'.} \emph{Can one use the degree-reverse-lexicographic order for elimination?}

This paper answers all of the above questions affirmatively.

\smallskip
Geometrically speaking, we are interested in the Zariski closure of the image of an affine subscheme $X \subset \mathbb{A}^n_B$ (over a base ring $B$) under the affine projection morphism $\pi: \mathbb{A}^n_B \twoheadrightarrow \Spec B$.
The idea employed in this paper is to factor\footnote{Such a factorization is a standard technique in intersection theory, e.g., the starting point of Grothendieck's proof of the Grothendieck-Riemann-Roch theorem.} $\pi$ over $\mathbb{P}^n_B$
\begin{center}
  \begin{tikzpicture}
  \coordinate (r) at (2.3,0);
  \coordinate (d) at ( 0,-1.3);
  
  \node(X) {$X$};
  \node (C) at ($(X)+(r)$) {$\overline{X}$};
  \node (pX) at ($(C)+1.3*(r)$) {$\overline{\pi}(\overline{X})=\overline{\pi(X)}$};
  \node (A) at ($(X)+(d)$) {$\mathbb{A}^n_B$};
  \node (P) at ($(A)+(r)$) {$\mathbb{P}^n_B$};
  \node (B) at ($(P)+1.3*(r)$) {$\Spec B$,};
  
  \draw[right hook-stealth'] (X) -- (C); 
  \draw[-doublestealth] (C) -- (pX);
  \draw[right hook-stealth'] (A) -- node[above]{\footnotesize{open}} (P);
  \draw[-doublestealth] (P) -- node [above] {$\overline{\pi}$} (B);
  \draw[right hook-stealth'] (X) -- (A);
  \draw[right hook-stealth'] (C) -- (P);
  \draw[right hook-stealth'] (pX) -- (B);
  
  \node at ($(pX)+1/2*(d)+1.5*(r)$) {\footnotesize (the vertical morphisms are closed)};
  
  \end{tikzpicture}
\end{center}
and to use the fact that the projective closure $\overline{X}$ (of $X$ in $\mathbb{P}^n_B$) has no components in the hyperplane at infinity $H := \mathbb{P}^n_B \setminus \mathbb{A}^n_B$ in order to represent $\overline{X}$ (non-canonically) as the Zariski closure of the set-theoretic difference between a projective scheme $Y \supset X$ and $H$, where $Y \subset \mathbb{P}^n_B$ is the vanishing set of the homogenization of a generating set of the vanishing ideal $\mathcal{I}(X)$.

\smallskip
Translating the projective closure from geometry to algebra, our approach to elimination of ideals (or submodules, cf.\ \cref{subsection_submodules}) in polynomial rings uses homogenization.
Homogenizing an ideal $I = \mathcal{I}(X)$ in the relative case ($B$ not a field) usually relies on computing a Gröbner basis w.r.t.\  a block-elimination order (cf., e.g., \cite[Prop.~15.31]{Eis}), and it is precisely this order which we want to offer an alternative to.
Instead, we show that the homogenized ideal $I^h$ can be computed by homogenizing any generating set of $I$ followed by a saturation w.r.t.\  the new homogenizing indeterminate.
The original elimination of the ideal $I$ can now be computed by evaluating certain indeterminates to zero in a generating set of $I^h$.

More formally, let $R := B[x_1,\ldots, x_n]$ be a polynomial ring over the commutative ring $B$.
We show how to compute the intersection $I\cap B$ for an ideal $I\unlhd R$ given by a generating set $G\subset I$, i.e., how to eliminate the indeterminates $x_1,\ldots, x_n$, without using a block-elimination order.
We base the alternative algorithm on the observation that $I \cap B$ is equal to the intersection of $B$ with the homogenization of $I$.
More precisely, for $f \in R$ denote by $f^h\in S := B[x_0,x_1,\ldots, x_n]$ the homogenization of $f$ w.r.t.\  the new indeterminate $x_0$.
The above mentioned equality reads $I \cap B = I^h \cap B$ (\Cref{lemm:intersection0}), where $I^h = \langle f^h \mid f \in I \rangle \unlhd S$ denotes the homogenization of $I$.
The homogenization $I^h$ can be computed using the equality $I^h = \langle g^h \mid g \in G \rangle: x_0^\infty$ (\Cref{prop:sat0}).
Finally, we obtain a generating set of the elimination ideal $I \cap B = I^h \cap B \unlhd B$ by simply evaluating $x_0,\ldots,x_n$ to zero in a generating set of $I^h$ (\Cref{lemm:degreezeropart}).

The above statements are proved in \Cref{sec:correctness} with no further assumption on the unital base ring $B$ other than being commutative.
In \Cref{sec:alt_proofs} we give alternative, more conceptual proofs of \Cref{prop:sat0} and \Cref{lemm:intersection0}.

Using homogenization to transfer a given computational task to an easier one is a common principle (see, e.g., \cite[Tutorial 51]{KR2}).
In contrast to the usual dehomogenization of the result of a homogeneous computation, we take global sections by simply evaluating indeterminates to zero.

\smallskip
Algorithmically speaking, the only computationally non-trivial step in the above sketched algorithm is the saturation w.r.t.\  $x_0$.
It is well-known that saturations can be computed by iterating ideal quotients, each of which can be computed using syzygies, for certain classes of base rings\footnote{for which the iterated ideal quotients stabilizes after finitely many steps.} $B$ (cf.~\cite[\S3.8]{AL} or \cite[\S15.10.6]{Eis}):

In case $B = k[b_1,\ldots,b_m]$ is a polynomial ring over a computable field $k$ we can use any term order of $S=k[b_1,\ldots,b_m,x_0,x_1,\ldots,x_n]$ to compute syzygies.
This very convenient setting allows performing computations in one big ring, as $B$ receives no special treatment in the Gröbner basis algorithm.
Similarly, if $B = k[b_1,\ldots,b_m]/L$ we can simply add any set of generators of $L$ to $I$ and proceed as above.

For the computability of syzygies over the ring $S=B[x_0,\ldots,x_n]$ it suffices to assume that $B$
 has effective coset representatives \cite[Theorem~4.3.15]{AL}.
A ring $B$ is said to have \emph{effective coset representatives} if for every ideal $L\unlhd B$ we can algorithmically find a set $T$ of coset representatives of $B/L$, such that for every $b \in B$ we can compute the unique $t \in T$ with $b+L=t+L$.
It is immediate that computable fields and the ring of integers $\Z$ have effective coset representatives.
Furthermore, if $B$ has effective coset representatives then also its residue class rings and polynomial rings there over.

We do not claim that our saturation-based algorithm can outperform the direct use of block-elimination orders.
In fact, our straightforward implementation in \texttt{homalg} (see \cite{homalg-project}) was significantly slower than direct elimination using a block-elimination order, where the Gröbner bases computations in both cases were delegated to \textsc{Singular} \cite{singular412}.
However, our implementation relies on \textsc{Singular}'s general procedures to compute syzygies, which are a priori not optimized for saturation with respect to a single indeterminate.

\section{Correctness of the algorithm}\label{sec:correctness}

Let $B$ be a unitial commutative ring\footnote{We do not need that $B$ is a polynomial ring to prove correctness.}, $R := B[x_1,\ldots,x_n]$ and $S := B[x_0,x_1,\ldots,x_n]$ polynomial rings over $B$, 
We view $S$ as a graded ring over $B$ equipped with its standard grading, i.e., $\deg(x_i) = 1$ for all $i=0,\ldots,n$, and $\deg(b) = 0 $ for all $b \in B \setminus \{0\}$.

\subsection{Homogenization}
We present an algorithm to compute homogenizations by using an arbitrary monomial order.
This algorithm is used as a subalgorithm for the elimination algorithm using an arbitrary monomial order in \cref{elimination}.

For $f \in R$ denote by $f^h$ the \emph{homogenization} of $f$ in $S$ w.r.t.\  the indeterminate $x_0$:
\[
  f^h := x_0^{\deg f} f \left( \frac{x_1}{x_0}, \ldots, \frac{x_n}{x_0} \right) \in S \mbox{.}
\]

Note that possible different degrees of the summands or cancelation of leading terms prevent the homogenization $f\mapsto f^h$ from being additive.
However, as a direct consequence of the definition one gets a slight generalization of additivity.
\begin{rmrk}[Generalized additivity] \label{rmrk:additivity}
  Let $f = \sum_i f_i \in R$ and $m := \max_i\{\deg f_i\}$.
  Then
  \[
    x_0^{m-\deg f} f^h = \sum_i x_0^{m-\deg f_i} f_i^h \mbox{.}
  \]
\end{rmrk}
If $B$ is a domain then the homogenization is multiplicative\footnote{This follows from $\deg(f g) = \deg(f)+\deg(g)$ if $B$ is a domain. If $B$ is not a domain then $\deg(f g) \leq \deg(f)+\deg(g)$: consider, e.g., the square of $2x+1$ in $(\Z/4\Z)[x]$.}; more generally:
\begin{rmrk}[Generalized multiplicativity] \label{rmrk:multiplicativity}
  Let $f, g\in R$. Then
  \[
    x_0^\ell (f g)^h = f^h g^h  \quad\mbox{with}\quad \ell := \deg f 
    + \deg g - \deg (f g) \geq 0 \mbox{.}
  \]
  If $B$ is a domain then $\ell = 0$, trivially.
\end{rmrk}
 
On the other hand the \emph{dehomogenization}
\[
  \phi : S \to R, \begin{cases} x_0 \mapsto 1, \\ x_i \mapsto x_i  & i \neq 0\end{cases}
\]
is an epimorphism of $B$-algebras with kernel $\langle x_0 - 1 \rangle$.
Moreover, $\phi(r^h) = r$ for all $r \in R$ and $\phi(s)^h=s$ for all homogeneous $s \in S$ with $s \notin \langle x_0 \rangle$.

The \emph{homogenization} of an ideal $I \unlhd R$ is the homogeneous ideal $I^h := \langle f^h \mid f \in I \rangle \unlhd S$, generated by the homogenization of all elements of $I$.
We now show that homogenized ideals are saturated 
w.r.t.\  $x_0$.
In \Cref{prop:sat0} we show the converse, namely that $I^h$ can be computed by such a saturation.

\begin{lemm}\label{lemm:homogenization_is_saturated}
  The homogenization $I^h\unlhd S$ of an ideal $I\unlhd R$ is saturated w.r.t.\  $x_0$, i.e.,
  \[
    I^h = I^h : x_0^\infty := \{ f \in S \mid \exists \ell \in \N_0: x_0^\ell f \in I^h \} \mbox{.}
  \]
\end{lemm}
This statement is the manifestation of the geometric fact that the closure $V(I^h) \subset \mathbb{P}^n_B$ of the vanishing set $V(I) \subset \mathbb{A}^n_B \subset \mathbb{P}^n_B$ does not include components in the hyperplane $V(x_0) \subset \mathbb{P}^n_B$ at infinity.
\begin{proof}
  It suffices to prove the inclusion $I^h \supset I^h : x_0^\infty$.
  Let $f \in I^h : x_0^\infty$ homogeneous.
  We can assume that $f \notin \langle x_0 \rangle$, since such elements generate $I^h : x_0^\infty$.
  Then there exists a nonnegative integer $\ell$ such that $x_0^\ell f \in I^h$, i.e., there exist $f_i \in I$ and homogeneous $s_i \in S$ such that $x_0^\ell f = \sum_i s_i f_i^h$.
  Applying the dehomogenization $\phi$ yields $\phi(f) = \phi(x_0^\ell f) = \sum_i \phi(s_i) f_i \in I$.
  As $\phi(f)\in I$ we finally have $f=\phi(f)^h\in I^h$, where the last equality holds since $f \notin \langle x_0 \rangle$ by the assumption.
\end{proof}

Homogenizing a generating set of an ideal does not generally yield a generating set of the homogenized ideal (cf.~\Cref{exmp:saturation}).
However, the necessary saturatedness condition from \Cref{lemm:homogenization_is_saturated} suffices in the following sense:

\begin{prop}\label{prop:sat0}
Let $G$ be a generating set of an ideal $I\unlhd R$.
Then
\[
   I^h = J : x_0^\infty
\]
for the homogeneous ideal $J := \langle g^h \mid g \in G \rangle \subset I^h \unlhd S$ generated by the homogenized elements of $G$.
\end{prop}
Geometrically speaking, the vanishing set $V(J)$ (of the homogenization of a generating set of $I$) coincides with $V(I^h)$ up to possible extra components of $V(J)$ which are contained in the hyperplane $V(x_0)$ at infinity (see \Cref{exmp:saturation}).

\begin{proof}
  The inclusion $ I^h\supseteq J : x_0^\infty$ follows from the inclusion $I^h\supseteq J$ and \Cref{lemm:homogenization_is_saturated}.
  
  We now prove the reverse inclusion.
  Consider an element $f^h$ for $f \in I$ (by definition $I^h$ is generated by elements of this form).
  Since $G$ is a generating set of $I$ we can write $f = \sum_{g\in G'} r_g g$ for some $r_g \in R$ and a finite subset $G' \subseteq G$.
  Set $m := \max\{ \deg r_g + \deg g \mid g\in G' \}$.
  Then homogenizing the above $R$-linear combination of $f$ yields the $S$-linear combination
  \begin{align*}
    x_0^{m - \deg f} f^h &= \sum_{g\in G'} x_0^{m-\deg(r_g g)} (r_g g)^h & \mbox{(\Cref{rmrk:additivity})\phantom{.}} \\
    &= \sum_{g\in G'} x_0^{m-(\deg r_g + \deg g)} r_g^h g^h \in J & \mbox{(\Cref{rmrk:multiplicativity}).}
  \end{align*}
  Hence, $f^h\in J : x_0^\infty$.
\end{proof}

Note that if our base ring $B$ is a field, proofs of \Cref{lemm:homogenization_is_saturated} and \Cref{prop:sat0} are given in \cite[Corollary 4.3.7, Corollary 4.3.8.a]{KR2}.
See also \cite[Exer.~15.40]{Eis}.

\subsection{Examples}
If $B = k[b_1,\ldots,b_m]$ is a polynomial ring over a field $k$, then the above saturation can be computed by Gröbner bases algorithms with respect to any monomial order.

Homogenizing a Gröbner basis of the ideal $I \unlhd R$ w.r.t.\  the \emph{block-elimination order} $b_i \prec x_j$ already yields a generating set of $I^h \unlhd S$, i.e., there is no need for subsequent saturation (cf.~\cite[Prop.~15.31]{Eis} or \cite[Exer.~1.6.19]{AL}).

However, in our alternative approach to homogenization using saturation (\Cref{prop:sat0}) we are free in choosing a monomial order.
One could proceed as follows:
Use the \emph{degree-reverse lexicographical order} (degrevlex) in $S$ to compute the saturation, as it is the fastest order on average.
When making \Cref{prop:sat0} algorithmic, we even suggest to start with a Gröbner basis $G$ of $I$ w.r.t.\  degrevlex in $R$ (since then for $B=k$ no subsequent saturation is required\footnote{This is \cite[Theorem 8.4.4 and Exer.~8.4.5]{CLO}, a special case of \cite[Prop.~15.31]{Eis}.}):

\begin{exmp} \label{exmp:saturation}
  Consider $B=k=\Q$, $R=\Q[x_1,x_2]$, $G := \{ x_1^2, x_2 - x_1^2 \} \subset R$, and $I = \langle G \rangle$.
  Then $J = \langle G^h \rangle = \langle x_1^2, x_0 x_2 - x_1^2 \rangle = \langle x_1^2, x_0 x_2 \rangle \subsetneq \langle x_1^2, x_2 \rangle = I^h$.
  However, the saturation $J : x_0^\infty = I^h$, complying with \Cref{prop:sat0}.
  If we would have started with the Gröbner basis $G' := \{ x_1^2, x_2 \}$ of $I$ the saturation step would've been obsolete.
\end{exmp}

Unless $B=k$, the next example shows that homogenizing a Gröbner basis of $I$ w.r.t.\ degrevlex (contrary to the above mentioned block-elimination order) does not generally yield a generating set of $I^h$ and a subsequent saturation is necessary.

\begin{exmp} \label{exmp:twisted_cubic}
  Consider $B=\Q[b_1,b_2,b_3]$, $R = B[t]$, and $I = \langle b_i - t^i  \mid i=1,2,3 \rangle \lhd R$ the ideal defining the twisted cubic.
  A Gröbner basis of $I$ w.r.t.\  degrevlex is $G := \{ b_1 - t, t^2 - b_2, b_2 t - b_3, b_2^2 - b_3 t \}$.
  Then $J = \langle b_1 s - t, t^2 - b_2 s^2, b_2 t - b_3 s, b_2^2 s - b_3 t \rangle \lhd S = B[s,t]$ and for obvious degree reasons $J$ cannot contain elements of degree $0$ in $s$ and $t$, i.e., $J \cap B = \{ 0 \}$.
  This example shows that homogenizing a Gröbner basis of $I$ w.r.t.\  the degrevlex does not generally yield a generating set of $I^h = J : \langle s \rangle^\infty = \langle b_1 s - t, b_1 b_2 - b_3, b_1^2 - b_2 \rangle$.
\end{exmp}

\subsection{Elimination via homogenization}\label{elimination}
Homogenizing an ideal $I\unlhd R$ does not add additional elements from $B$ to the ideal, allowing us to use homogenization for elimination, which is our main goal:

\begin{lemm}\label{lemm:intersection0}
  For an ideal $I\unlhd R$ we have $I^h \cap B = I \cap B$.
\end{lemm}
\begin{proof}
  The homogeneity of elements of $B$ yields the inclusion $I^h \cap B \supset I \cap B$.
  So we only need to show the inclusion $I^h \cap B \subset I$.
  To this end let $b \in I^h \cap B$.
  As $b\in I^h$ there exist finitely many $f_i \in I$ and homogeneous $s_i \in S$ such that $b = \sum_i s_i f_i^h$.
  Applying the dehomogenization $\phi$ yields $b = \phi(b) = \sum_i \phi(s_i) \phi(f_i^h) = \sum_i \phi(s_i) f_i \in I$.
\end{proof}

Since $B$ is the degree zero part of $S$, this intersection $I^h \cap B$ is just the degree zero part of $I^h$.
We shortly explain on how to read off this degree zero part of any homogeneous ideal:

\begin{rmrk}
Let $H \unlhd S$ be a homogeneous ideal.
By definition, a polynomial is contained in $H$ iff all of its homogeneous parts are.
Hence, given a generating set $G$ of $H$, the set $G':=\{p_i\mid p\in G,i\in\Z_{\ge0}\}$ of the homogeneous parts of $G$ is also a generating set of $H$.
Removing a polynomial $p \in G'$ of degree $d$ from $G'$ does not affect the parts of degree smaller than $d$, i.e., $H$ is equal to $\langle G'\setminus\{p\}\rangle\unlhd S$ in all degrees smaller $d$.
In particular, $G'_0:=\{p\in H' \mid \deg(p)=0\} = \{ p_{\mid x_0=0,\ldots,x_n=0} \mid p \in G \}$ is a generating set of the degree $0$ part $H\cap B$.
\end{rmrk}

The following corollary can be seen as an elimination in the case of projective fibers, and it is much simpler than the usual affine elimination.

\begin{coro}\label{lemm:degreezeropart}
  Let $I\unlhd R$ be an ideal.
  Evaluating the indeterminates $x_0,\ldots,x_n$ to zero in a generating set of $I^h \unlhd S$ yields a generating set of $I^h\cap B = I \cap B$.
\end{coro}

\Cref{prop:sat0} gives an algorithm to compute a generating set for $I^h$ as a Gröbner basis of $J : x_0^\infty$ by saturating $J$ w.r.t.\ the indeterminate $x_0$.
We again emphasize that we do not need any elimination order, as a saturation can be computed by iterated syzygies w.r.t.\  any order.
By \Cref{lemm:intersection0} and \cref{lemm:degreezeropart} we can then compute the elimination $I \cap B$ by evaluating the $x_i$ to zero in $I^h$.

\begin{exmp}
Continuing \Cref{exmp:twisted_cubic} we get $I \cap B  = I^h \cap B = \langle b_1 b_2 - b_3, b_1^2 - b_2 \rangle \lhd B$ (with Gröbner basis $\{ b_1 b_2 - b_3, b_1^2 - b_2, b_2^2 - b_1 b_3 \}$).
\end{exmp}

\subsection{Generalization to submodules}\label{subsection_submodules}
  All statements and above proofs easily generalize to the case of submodules of finite rank free modules.
  For $\ell \in \N_0$ let $I \leq R^\ell$ be an $R$-submodule generated by the finite set $G \subset R^\ell$.
  For $f=(f_1,\ldots, f_\ell) \in R^\ell$ define $f^h = (f_1,\ldots, f_\ell)^h := (x_0^{i_1} f_1^h, \ldots, x_0^{i_\ell} f_\ell^h)$, where $i_j$ is determined by the equality $i_j + \deg(f_j) = \deg(x_0^{i_j} f_j) = \max\{ \deg(f_k) \mid k = 1,\ldots,\ell \}$.
  The homogenization of $I$ is the graded $S$-submodule $I^h = \langle f^h \mid f \in I \rangle \leq S^\ell$ and the saturation of $J := \langle f^h \mid f \in G \rangle$ by $x_0$ is $J : x_0^\infty = \{ f \in S^\ell \mid \exists k \in \N_0: x_0^k f \in J \}$ w.r.t.\  $x_0$.
  Elimination now means computing the intersection $I \cap B^\ell = I^h \cap B^\ell$ with $B^\ell := \{ (r_1,\ldots, r_\ell) \mid r_i \in B \} \subseteq R^\ell$, which reduces to evaluating the indeterminates $x_0,\ldots, x_n$ to zero in a generating set of $I^h$.

\appendix

\section{Alternative proofs} \label{sec:alt_proofs}

In this section we give some alternative, more conceptual proofs of \Cref{prop:sat0} and \Cref{lemm:intersection0}.

\begin{proof}[Alternative proof of \Cref{prop:sat0}]
 Define the $B$-algebra homomorphisms
 \begin{center}
  \begin{tikzpicture}
    \coordinate (r) at (1.8,0);
    \coordinate (d) at (0,-0.6);
    
    \node (R) {$R$};
    \node (Rx) at ($(R)+2*(r)$) {$S_{x_0}$};
    \node (S) at ($(Rx)+2*(r)$) {$S$.};
    \node (Re) at ($(R) + (d) + 0.7*(r)$) {$x_i$};
    \node (Rxe1) at ($(Rx) + (d) - 0.7*(r)$) {$\frac{x_i}{x_0}$};
    \node (Rxe2) at ($(Rx) + (d) + 0.7*(r)$) {$\frac{x_i}{1}$};
    \node (Se) at ($(S) + (d) - 0.7*(r)$) {$x_i$};
    
    \draw[right hook-stealth'] (R) -- node[above] {$\alpha$} (Rx);
    \draw[left hook-stealth'] (S) -- node[above] {$\beta$}  (Rx);
    \draw[draw = none] (Re) -- node{$\mapsto$} (Rxe1);
    \draw[draw = none] (Se) -- node{$\mapsfrom$} (Rxe2);
  \end{tikzpicture}
  \end{center}
  We write $\alpha_{\ast}$ and $\beta_{\ast}$ for the operators that send an ideal to its extension with respect to $\alpha$ and $\beta$, respectively.
  Now, given any subset of elements $F \subset R$, we have
  \begin{align*}
   \beta^{-1}\alpha_{\ast}\big( \langle f \mid f \in F \rangle_R \big) &= \beta^{-1}\big( \langle f\left(\frac{x_1}{x_0}, \dots \frac{x_n}{x_0}\right) \mid f \in F \rangle_{S_{x_0}} \big) \\
   &= \beta^{-1}\big( \langle f^h \mid f \in F \rangle_{S_{x_0}} \big) \\
   &=\beta^{-1}\beta_{\ast} \big( \langle f^h \mid f \in F \rangle_{S} \big) \\
   &= \langle f^h \mid f \in F \rangle_S : x_0^\infty.
  \end{align*}
  On the one hand, setting $F = I$ yields
	\begin{align*}
	\beta^{-1}\alpha_{\ast}( I ) = I^h : x_0^\infty = I^h,
	\end{align*}
  where the second equality follows from \Cref{lemm:homogenization_is_saturated}, and on the other hand, setting $F = G$ gives us
  \begin{align*}
   \beta^{-1}\alpha_{\ast}( I ) = \langle g^h \mid g \in G \rangle : x_0^\infty = J : x_0^\infty.
  \end{align*}
\end{proof}

\begin{rmrk}\label{rmrk:algebras}
  Let $B$ be a commutative ring and let $\overline{R}$, $\overline{S}$ be $B$-algebras with
  structure morphisms $\rho: B \rightarrow \overline{R},\  b \mapsto b\cdot 1_{\overline{R}}$
  and $\sigma: B \rightarrow \overline{S},\  b \mapsto b \cdot 1_{\overline{S}}$.
  Assume that there exists  a $B$-algebra homomorphism $\overline{\phi}: \overline{S} \rightarrow \overline{R}$. 
  Then
  \[
   \ker \sigma \subset \ker \rho \mbox{.}
  \]
\end{rmrk}
\begin{proof}
  The statement that $\overline{\phi}$ is a morphism of $B$-algebras means that $\overline{\phi} \circ \sigma = \rho$ and hence $\ker (\overline{\phi} \circ \sigma) = \ker \rho$.
  This implies the claim.
  
  The categorical formulation would be
  \begin{center}
  \begin{tikzpicture}
    \coordinate (r) at (1.5,0);
    \coordinate (d) at (0,-1.3);
    
    \node (B1) {$B$};
    \node (B2) at ($(B1)+(d)$) {$B$};
    \node (R) at ($(B1)+(r)$) {$\overline{R}$};
    \node (S) at ($(B2)+(r)$) {$\overline{S}$};
    \node (kerrho) at ($(B1)-(r)$) {$\ker \rho$};
    \node (kersigma) at ($(B2)-(r)$) {$\ker \sigma$};
    
    \draw[-stealth'] (B1) -- node[above] {$\rho$} (R);
    \draw[-stealth'] (B2) -- node[above] {$\sigma$}  (S);
    \draw[right hook-stealth'] (kerrho) -- node[above] {$\kappa_\rho$} (B1);
    \draw[right hook-stealth'] (kersigma) -- node[above] {$\kappa_\sigma$}  (B2);
    \draw[-stealth'] (S) -- node[right] {$\overline{\phi}$} (R);
    \draw[double distance=0.2em] (B1) -- (B2);
    \draw[left hook-stealth'] (kersigma) -- node[left] {$\iota$} (kerrho);
  \end{tikzpicture}
  \end{center}
  to construct the mono $\iota$ as the unique lift of $\kappa_\sigma$ along the kernel embedding $\kappa_\rho$, which exists since $\rho \circ \kappa_\sigma = 0$ by the commutativity of the right square.
  Since $\kappa_\rho \circ \iota = \kappa_\sigma$ is a mono it follows that $\iota$ is itself a mono.
\end{proof}

\begin{proof}[Alternative proof of \Cref{lemm:intersection0}]
 We define the $B$-algebras $\overline{R} := R/I$ and $\overline{S} := S/I^h$,
 and denote their structure morphisms by $\rho$, $\sigma$, respectively.
 Their kernels are given by $I \cap B$ and $I^h \cap B$, respectively,
 and the inclusion $I^h \cap B \supset I \cap B$ was the easy inclusion in the first proof of \Cref{lemm:intersection0}.
 Furthermore, the dehomogenization $\phi: S \rightarrow R$ induces a $B$-algebra homomorphism
 $\overline{\phi}: \overline{S} \rightarrow \overline{R}$.
 We can now apply \Cref{rmrk:algebras} which completes the proof.
\end{proof}

\def\cprime{$'$} \def\cprime{$'$} \def\cprime{$'$} \def\cprime{$'$}
  \def\cprime{$'$}
\providecommand{\bysame}{\leavevmode\hbox to3em{\hrulefill}\thinspace}
\providecommand{\MR}{\relax\ifhmode\unskip\space\fi MR }
\providecommand{\MRhref}[2]{%
  \href{http://www.ams.org/mathscinet-getitem?mr=#1}{#2}
}
\providecommand{\href}[2]{#2}

\end{document}